\documentclass[12pt,twoside]{amsart}

\usepackage{microtype}
\usepackage[OT1]{fontenc}
\usepackage{textcomp}
\usepackage{type1cm}
\usepackage{amssymb}
\usepackage{mathtools}
\usepackage{esint}
\usepackage{nicefrac}
\usepackage{enumerate}
\usepackage[left=3.2cm,top=3.8cm,right=3.2cm]{geometry}

\usepackage{cite}

\usepackage[active]{srcltx}
\usepackage{verbatim}

\usepackage[pdftex]{graphicx}
\usepackage{caption}
\usepackage{subfigure}

\usepackage{xcolor}
\usepackage[pagebackref]{hyperref}
\hypersetup{
    colorlinks,
    linkcolor={red!60!black},
    citecolor={blue!60!black},
   urlcolor={blue!90!black}
}

\usepackage{tikz}
\usetikzlibrary{trees,patterns,calc}
\usepackage{ifthen}
\usepackage{pgfplots}
\pgfplotsset{compat=newest}

\newcommand\fractalpercolation[4]{
	   \pgfmathsetseed{#4}
	      \pgfmathsetmacro{\p}{#1} 
	      \pgfmathsetmacro{\b}{#2} 
	      \pgfmathsetmacro{\l}{#3} 
	      \pgfmathparse{int(\l)}
	      \fill[black] (0.01,0.01) rectangle (0.99,0.99); 
	      \ifnum\pgfmathresult=0
	      
	      \else
	      \foreach \k in {1,...,\l}{ 
		\pgfmathsetmacro{\n}{pow(\b,\k)-1}
		\pgfmathsetmacro{\t}{1/(pow(\b,\k))}
		\begin{scope}[scale=\t]
		  \foreach \i in {0,...,\n}{
		  \foreach \j in {0,...,\n}{
			\pgfmathparse{(rand+1)/2>\p ? int(1) : int(0)} 
			\ifnum\pgfmathresult=1 
			\fill [white] (\i,\j) rectangle ++(1,1); 
			\fi
		  }}
		\end{scope}
	      }
	      \fi
}

\numberwithin{equation}{section}

\theoremstyle{plain}
\newtheorem{theorem}{Theorem}[section]

\newtheorem{proposition}[theorem]{Proposition}
\newtheorem{lemma}[theorem]{Lemma}

\theoremstyle{remark}
\newtheorem{remark}[theorem]{Remark}
\newtheorem{remarks}[theorem]{Remarks}

\newtheorem*{ack}{Acknowledgement}

\theoremstyle{definition}

\newcommand{\PP}{\mathcal{P}}

\newcommand{\CC}{\mathcal{C}}

\newcommand{\QQ}{\mathcal{Q}}

\newcommand{\R}{\mathbb{R}}

\newcommand{\Z}{\mathbb{Z}}
\newcommand{\N}{\mathbb{N}}
\newcommand{\EE}{\mathbb{E}}

\newcommand{\iii}{\mathtt{i}}
\newcommand{\jjj}{\mathtt{j}}
\newcommand{\kkk}{\mathtt{k}}
\newcommand{\ppp}{\mathtt{p}}
\newcommand{\mmm}{\mathtt{m}}

\newcommand{\roo}{\varrho}

\renewcommand{\epsilon}{\varepsilon}

\newcommand{\BB}{\mathcal{B}}

\DeclareMathOperator{\dimh}{dim_H}

\DeclareMathOperator{\dima}{dim_A}

\DeclareMathOperator{\cdima}{\mathcal{C}dim_A}
\DeclareMathOperator{\cdimh}{\mathcal{C}dim_H}
\DeclareMathOperator{\gcdimh}{\mathcal{G}\mathcal{C}dim_H}

\DeclareMathOperator{\diam}{diam}

\begin{document}

\title{Fractal percolation and quasisymmetric mappings}

\author{Eino Rossi}
\address{
        Department of Mathematics and Statistics, University of Helsinki \\
        P.O. Box 68  (Pietari Kalmin katu 5) \\
        00014 University of Helsinki, Finland}
\email{eino.rossi@gmail.com}

\author{Ville Suomala}
\address{
        Department of Mathematical Sciences, University of Oulu\\
        P.O. Box 8000 (Pentti Kaiteran katu 1)\\ 
        FI-90014 University of Oulu, Finland}
\email{ville.suomala@oulu.fi}

\thanks{ER acknowledges the supports of CONICET, the Finnish Academy of Science and Letters, Mittag-Leffler institute, and the University of Helsinki via the project
Quantitative rectifiability of sets and measures in Euclidean Spaces and Heisenberg groups (project No.7516125)\\
VS acknowledges support from the Academy of Finland CoE in Analysis and Dynamics research.\\
We thank the Mittag-Leffler institute and the organizers of the  Fractal Geometry and Dynamics program, where this project started. VS also acknowledges the Finnish Academy of Science and Letters for covering the costs of the visit.
}

\date{\today}

\begin{abstract}
  We study the conformal dimension of fractal percolation and show that, almost surely, the conformal dimension of a fractal percolation is strictly smaller than its Hausdorff dimension.
\end{abstract}

\maketitle

\section{Introduction}
Dropping the dimension of a metric space $X$ by a quasisymmetric mapping is a popular and challenging question in analysis. Especially, this question has been studied for many deterministic fractal sets, such as  Sierpi\'nski carpets and other self-similar sets, and also Bedford-McMullen carpets and other self-affine sets \cite{MackayTyson2010,DavidSemmes1997,TysonWu2006,Mackay2011,KaenmakiOjalaRossi2018}. In this note, we study the conformal dimension for certain random sets. Given a random set $E\subset\R^d$ with almost sure Hausdorff dimension $s\in (0,d]$, it is natural to ask if $E$ is almost surely minimal for conformal dimension, and if not, what is the ``almost sure" or ``expected" conformal dimension of $E$.
We are not aware of any results of this kind for continuous random sets, but related embedding questions have been investigated for discrete sets, especially for Lipschitz embeddings of independent Bernoulli random sets in $\Z$ and $\Z^d$, see \cite{BasuSly2014, BSS2018}.
\subsection{Conformal dimension}
A homeomorphism $\eta\colon [0,\infty) \to [0,\infty)$ is called a control function. A homeomorphism $f$ between metric spaces $(X,d)$ and $(Y,\roo)$ is called $\eta$-quasisymmetric 
if
\begin{equation}
 \label{eq:qs_def}
 \frac{ d( f(x) , f(y) ) }{ d( f(x) , f(z) ) } \leq \eta \left( \frac{ \roo(x,y) }{ \roo(x,z) } \right)
\end{equation}
holds for all $x,y,z \in X$ with $x\neq z$. We say that $f$ is quasisymmetric (qs for short), if it is $\eta$-quasisymmetric for some control function $\eta$. Quasisymmetric mappings preserve certain geometric properties like doublingness, uniform perfectness and so on, but unlike Lipschitz mappings they may distort common notions of dimension such as Hausdorff dimension. For standard properties of quasisymmetric maps we refer to \cite{TukiaVaisala1980,Heinonen2001,MackayTyson2010}.

If $f\colon X \to Y$ is $\eta_1$-quasisymmetric and $g\colon Y \to Z$ is $\eta_2$-quasisymmetric, then $g \circ f$ is $\eta_2 \circ \eta_1$-quasisymmetric and  $f^{-1}$ is $\eta'$-quasisymmetric, with $\eta'(t) = 1/\eta_1^{-1}(1/t)$. Therefore metric spaces $X$ and $Y$ are said to be quasisymmetrically-equivalent if there is a quasisymmetric $f\colon X \to Y$. Given a metric space $(X,d)$ a natural question is to describe the spaces that are quasisymmetrically-equivalent to $X$. This is known as the quasisymmetric uniformization problem. For example, a metric space is quasisymmetrically-equivalent to the middle thirds Cantor set if and only if it is compact, doubling, uniformly perfect, and uniformly disconnected \cite[Theorem 15.11]{DavidSemmes1997}. For more on quasisymmetric uniformization, we refer to the survey \cite{Bonk2006}.

The conformal (Hausdorff) dimension is a natural quasisymmetric invariant of a metric space. For a metric space $X$, its conformal dimension, $\cdimh X$, is the infimum of Hausdorff dimensions $\dimh Y$, among spaces that are qs-equivalent to $X$. That is
\[
 \cdimh X = \inf\{ \dimh f(X) : f\colon X \to f(X) \text{ is quasisymmetric}\,\}.
\]
If $X\subset \R^d$, then it is relevant to restrict to quasisymmetries $f\colon\R^d\to\R^d$.
 This leads to the definition of global conformal dimension: For $X\subset \R^d$, we define its global conformal dimension as
\[
 \gcdimh X = \inf \{ \dimh f(x) : f\colon \R^d \to \R^d \text{ is quasisymmetric}\,\}.
\]
(Recall that a homeomorphism $f\colon \R^d \to \R^d$ is a quasisymmetry if and only if it is quasiconformal). One can also consider the corresponding conformal dimensions for other notions of dimension (such as Assouad dimension, box-counting dimension, packing dimension), but we focus on the Hausdorff dimension. If $\dimh X = \cdimh X$ (resp. $\gcdimh X=\dimh X$), then we say that $X$ is minimal for (global) conformal dimension. 

\subsection{Conformal dimension of self-similar and self-affine sets}
Since self-similar sets satisfying the strong separation condition are uniformly disconnected (and uniformly perfect and compact), they are all qs-equivalent and thus have conformal dimension zero. Without the strong separation condition, the situation is very different. A model example is the classical Sierpi\'nski carpet $S_{3}\subset \R^2$ (see \cite[Section 4.3]{MackayTyson2010}). The carpet $S_3$ is obtained by dividing the unit square into $9$ equal sub-squares, removing the middle one and continuing the process in the remaining 8 squares iteratively. 
The resulting set $S_3$ is a self-similar set that satisfies the open set condition. It can be shown that $1 + \log_3 2 \leq \cdimh S_{3} < \dimh S_{3} = \log_3 8$, but the exact value of $\cdimh S_{3}$ is unknown. (For the strongest claimed bounds, see \cite{Kwapisz2017}.) The bound $\cdimh S_{3} < \dimh S_{3}$ follows by an Assouad dimension estimate (see \cite{Luukkainen1998} for the definition and properties of the Assouad dimension): It holds that $\dimh S_{3} = \dima S_{3}$, and, by passing to certain tangent sets and using  modulus estimates that we do not review here, it follows that $\cdima S_{3} < \dima S_{3}$ (see \cite[Example 6.2.3]{MackayTyson2010}). One can consider the same construction with any odd integer $p$ instead of $3$ dividing the unit square into $p^2$ equal squares of side-length $p^{-1}$ and removing the middle one. This results into a central carpet $S_p$, and one can deduce the corresponding results for $S_p$. It was shown by Bonk and Merenkov \cite{BonkMerenkov2013}, that central carpets $S_p$ and $S_q$ are never qs-equivalent for $p\neq q$.

In the above example it was crucial that $\dimh S_3 = \dima S_3$. This is a common property among self-similar sets satisfying a reasonable separation condition, for example, the open set condition. Self-affine sets $S$, on the other hand, typically satisfy $\dimh S < \dima S$ and thus a similar method does not work. However, it is still possible to study their conformal dimension. A Bedford-MacMullen carpet $S$ is obtained using a similar process as the central carpets $S_p$ except that the unit square is divided into $n\times m$ congruent sub-rectangles ($n>m$) and some of them are removed according to a given pattern which stays the same during the process. Mackay \cite{Mackay2011} verified the following dichotomy: If the retaining pattern contains an empty row and none of the rows is full, then $\cdima S = 0$ (and thus trivially also $\cdimh S =0$), otherwise $\cdima S = \dima S$. A similar in spirit, but incomparable class of of self-affine sets was recently studied by K\"aenmaki, Ojala, and Rossi \cite{KaenmakiOjalaRossi2018}.

\subsection{Main result}
In this paper, we focus on a well known family of random fractal sets, the fractal percolation. It is defined on the unit cube $[0,1]^d$ via two parameters $M\in\N_{\ge3}$ and $0<p<1$. We divide the unit cube into $M^d$ sub-cubes. Each of the sub-cubes survives with probability $p$ and perishes with probability $1-p$ independently of the other cubes. This process is then iterated in the surviving sub-cubes of all generations. The fractal percolation set $E$ consists of the points that survive all stages of the construction.  We recall the definition and the basic properties of fractal percolation in Section \ref{sec:fractalpercolation}. The main result of this paper implies that almost surely on non-extinction, fractal percolation is not minimal for conformal Hausdorff dimension:

\begin{theorem}
\label{thm:notminimal}
 Let $E=E_{M,p}(\omega)$ be the fractal percolation set in $\R^d$ and $p\in(0,1)$. Then there exists $t=t(d,M,p)$ so that $\cdimh E = t < \dimh E$ almost surely conditioned on the event $E\neq \emptyset$.
\end{theorem}

It is well known that, conditioned on non-extinction, the Hausdorff dimension of fractal percolation obtains a constant value almost surely (see \eqref{eq:dim_H_fp} below). By an application of a zero-one law (see Proposition \ref{thm:almostsureCdim}), also the conformal dimension of the fractal percolation takes a constant value, almost surely conditioned on non-extinction. Thus the task is to show that, conditioned on non-extinction, one can almost surely drop the dimension by a quasisymmetry. The random structure of the fractal percolation implies that fractal percolation does not satisfy uniform disconnectedness nor uniform perfectness, so unlike for well separated self-similar sets, there are no obvious ways to drop the dimension via quasisymmetric maps. Also, it almost surely holds that $\dimh E < \dima E = \cdima E = d$, so the methods from central carpets do not work either.

We prove Theorem \ref{thm:notminimal} in Section \ref{sec:cd} by constructing a quasisymmetry $f_\omega \colon E \to F$ for each realization $E = E(\omega)$ of the fractal percolation process, where $F=F(\omega)$ is a (random) subset of $\R^d$, and then showing that $\dim_H F<\dim_H E$ almost surely.  The main idea is quite simple: Suppose that $Q$ is an $M$-adic cube that survives the percolation. There is a fixed positive probability, independent of $Q$, that all the sub-cubes of $Q$ on the boundary die out. If this happens, it is possible to shrink the inner part of the cube by a contractive map. Iterating this over surviving cubes of all generations, yields a quasisymmetric map defined on $E$. We will show that around most points of $E$, this shrinking happens on so many scales that the induced quasisymmetry will decrease the Hausdorff dimension of the fractal percolation set $E$.   

In Section \ref{sec:gcd}, we prove Theorem \ref{thm:glob}, the global version of Theorem \ref{thm:notminimal}, by extending $f$ to a quasisymmetry $\R^d\to\R^d$.

\begin{ack}
 We thank John Mackay and the referees for useful comments.
\end{ack}

\section{Fractal percolation}\label{sec:fractalpercolation}

We recall the definition of fractal percolation in $\R^d$ as follows: To begin with, we fix a parameter $M\in\{3,4,\ldots\}$ and let $\Lambda=\{1,\ldots,M^d\}$. We first define a labelling of the $M$-adic sub-cubes of the unit cube $[0,1]^d$ using the alphabet $\Lambda$. We denote by $\QQ_n$ the family of closed $M$-adic sub-cubes of level $n$ of the unit cube $[0,1]^d$,
\[
 \QQ_n=\left\{\prod_{l=1}^d[i_l M^{-n},(i_{l}+1)M^{-n}]\,:\,0\le i_l\le M^{n}-1\right\}
\]
and let $\QQ=\cup_{n\in\N}\QQ_n$. Write $Q'\subset_N Q$ if there is $n\in\N$ such that $Q\in\QQ_n$, $Q'\in\QQ_{n+N}$ and $Q'\subset Q$.
Let us label the elements of $\QQ_1$ by the set $\Lambda=\{1,\ldots,M^d\}$ such that the symbols 
\[i\in\Lambda_B:=\{1,\ldots,M^d-(M-2)^d\}\] 
correspond to the cubes intersecting the boundary of $[0,1]^d$.
We fix this labelling throughout the paper and given any $Q\in\QQ_n$, we label the cubes $Q'\subset_1 Q$ using the same labelling based on the geometric location of $Q'$ inside $Q$. This induces a natural labelling of $\QQ_n$ by $\Lambda^n$.
\begin{figure}
\centering
\begin{tikzpicture}[scale=4]
 \fractalpercolation{0.7}{3}{1}{5}
\begin{scope}[shift={(1.3,0)}]
 \fractalpercolation{0.7}{3}{2}{5}
\end{scope}
\begin{scope}[shift={(2.6,0)}]
 \fractalpercolation{0.7}{3}{3}{5}
\end{scope}
%
\end{tikzpicture}
\label{fig:prcolation}
\caption{The first three stages of a sample of the fractal percolation process in dimension $2$ with parameters $p=0.7$ and $M=3$.}
\end{figure}
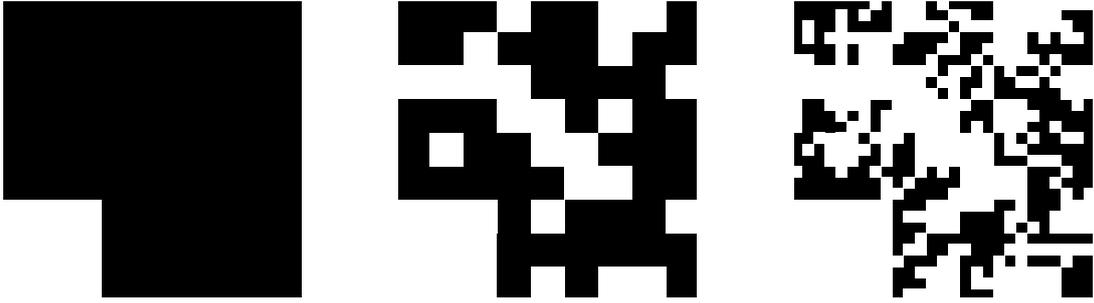

The main reason for the assumption $M > 2$ is that for $M=2$, $\Lambda_B=\Lambda$ since every cube touches the boundary of its parent. We leave it to the interested reader to check that all our results also hold for $M=2$. Basically, this reduces to the situation $M=4$ by gluing together two consecutive steps in the process.

We will use notations $\iii=i_1i_2\ldots,\jjj=j_1j_2\ldots$, etc. for words $\iii,\jjj\in\Lambda^{\N}$ and also for finite words $\iii,\jjj\in\Lambda^*=\bigcup_{n\in\N}\Lambda^n$. We denote by $\sigma$ the left shift, $\sigma(i_1i_2i_3\ldots)=i_2i_3\ldots$. If $\iii\in\Lambda^*$, we let $|\iii|$ stand for its length. The notation $\iii\prec\jjj$ means that $\jjj=\iii$ or  $\jjj=\iii\jjj'$ for some $\jjj'\in\Lambda^{\N}\cup\Lambda^*$ (Note that this is equivalent to saying that $Q_{\jjj}\subset_N Q_{\iii}$ for some $N\ge 1$ whenever $\jjj$ has finite length). By $\iii\wedge\jjj$ we mean the longest common beginning of the two words $\iii$ and $\jjj$. If $|\iii|\ge n$, let $\iii|_n=i_1\ldots i_n$ and $[\iii|_n] = \{\jjj\in\Lambda^\N : \jjj|_n = \iii|_n\}$. We say that a collection $\Lambda'\subset\Lambda^*\cup\Lambda^\N$ of words is incomparable if there are no words $\iii,\jjj\in\Lambda'$ such that $\iii\neq\jjj$ and $\iii\prec\jjj$.

Since $\QQ_n$ is in one to one correspondence with the alphabet $\Lambda^n$, we may define a natural projection $\Pi\colon\Lambda^\N\to[0,1]^d$ by setting
\[
 \Pi(\iii)=x_\iii\,,\text{ where }\{x_\iii\}=\bigcap_{n\in\N}Q_{\iii|_n}\,.
\]
Equivalently, $\Pi(\iii) = \lim_{n\to\infty} h_{ Q_{\iii|_n} }(0)$, where $h_Q$ is the homothety sending $[0,1]^d$ to $Q$. Given $0<p<1$, the fractal percolation set $E(\omega)$ is then defined by
\[E(\omega):=\Pi(T(\omega))\,,\]
where
\[T(\omega)=\{\iii\in\Lambda^\N\,:\,X_{\iii|_n}=1\text{ for all }n\in\N\}\,,\]
and $X_\iii$, $\iii\in\Lambda^*$ are independent Bernoulli random variables each taking value $1$ with probability $p$.
We denote by $\PP$ the law of the fractal percolation process. This is the unique Borel probability measure on $\{0,1\}^{ \Lambda^\N }$ satisfying
\[\PP\left(\prod_{n=1}^{|\iii|} X_{\iii|_n}=1\right) = p^{|\iii|} \text{ for all }\iii\in\Lambda^*\,.\]
It is well known that if $p>M^{-d}$, then almost surely, conditional on non-extinction (that is, $E\neq \varnothing$), it holds that
\begin{equation}\label{eq:dim_H_fp}
\dimh(E)=d+\frac{\log p}{\log M}\,.
\end{equation}
On the other hand, if $p\le M^{-d}$ then $E$ is almost surely empty, see for example \cite[Theorem 3.7.1]{BishopPeres2017}.

Let us also denote by $T_n\subset\Lambda^n$ the (random) set of $n$-words $\iii=i_1\ldots i_n$ such that
\[
 X_{\iii|_k}=1\text{ for all }k=1,\ldots,n\,.
 \]
By $\mathcal{B}_n$, we denote the sigma-algebra generated by  $\{X_{\iii}\,:\,|\iii|\le n\}$. In other words, $\BB_n$ is the sigma-algebra generated by the labelled finite trees $T_n \in \{0,1\}^{\Lambda^n}$.

Note that $T_n(\omega)$, $n\ge 1$ and $T(\omega)$ can be interpreted as trees by considering an edge between the vertices $\iii$ and $\jjj$ if $\jjj\in T_{|\jjj|}$,  $\iii\in T_{|\iii|}$ and either $\jjj = \iii k$ or $\iii = \jjj k$ for some $k\in\Lambda$. For convenience, we extend the natural projection for finite words as well. For $\iii\in\Lambda^*$, it holds that $\Pi([\iii]) = Q_\iii$, and so we set $\Pi(\iii) = h_{Q_\iii}(0)$. With this notation, we have the identity
\begin{equation}
 \label{eq:Pi_decomposition}
 \Pi(\iii\jjj) = \Pi(\iii) + M^{-|\iii|}\Pi(\jjj) = h_{Q_\iii}(\Pi(\jjj))
\end{equation}
for $\iii\in\Lambda^*$ and $\jjj\in \Lambda^\N \cup \Lambda^*$.

\begin{remark}\label{rem:sss}
The fractal percolation process is stochastically self-similar in the sense that conditional on $\iii\in T_n$, the sub-tree of $T$ rooted at $\iii$ has the same law as the original tree $T$. Transferred into $[0,1]^d$, this means that, conditional on $\iii\in T_n$, the set
\[
 h^{-1}_{Q_\iii}( \Pi(T(\omega) \cap [\iii]))
\] 
has the same law as $E$, for all $k\in\N$. Moreover, if $\Lambda'\subset\Lambda^*$ is a (finite) collection of incomparable words, the random sets $h^{-1}_{Q_\iii}(\Pi( T(\omega)\cap[\iii] ))$
are independent conditional on the event
\[
 \iii\in T_{|\iii|}\text{ for all }\iii\in\Lambda'\,.
\]
Note that $ \Pi(T(\omega) \cap [\iii] )$ is almost the same as $E\cap Q_\iii$, expect that $\partial Q_\iii \cap E$ may contain points of $E$ of the form $\Pi(\jjj)$ for some $\jjj\not\in [\iii|_n]$.	
\end{remark}

\begin{remarks}
\emph{a)} The fractal percolation set $E$ is a model example of a random fractal constructed from a Galton-Watson tree. In addition to fractal percolation, our main results can be extended to many other Galton-Watson random fractals, see Remark \ref{rem:final_remarks}.

\emph{b)} Many properties of the fractal percolation sets can be deduced from corresponding properties of the underlying Galton-Watson tree. For instance, the Hausdorff dimension of $E(\omega)$ can be directly deduced from the branching number of the tree $T(\omega)$, see \cite[\S 1.10]{LyonsPeres2016}. When it comes to conformal dimension, there is in general no relation between $\CC\dim E(\omega)$ and $\CC\dim(T(\omega))$; More precisely, a natural way to metrize  $T(\omega)$ and its boundary $\partial T(\omega)\subset \Omega^\N$ is to consider $0<\kappa<1$ and define $d_\kappa(\iii,\jjj)=\kappa^{|\iii\wedge\jjj|}$. Since for all $0<\kappa,\kappa'<1$, the spaces $(\partial T(\omega),d_\kappa)$ and $(\partial T(\omega),d_{\kappa'})$ are quasisymmetrically equivalent (the identity map yielding the obvious quasisymmetry), it follows that $\cdimh(T(\omega),d_\kappa)=0$.
\end{remarks}

\section{Conformal dimension of fractal percolation}\label{sec:cd}

We first observe that the conformal dimension of $E=E(\omega)$ is constant almost surely on $E\neq\varnothing$.

\begin{proposition}
 \label{thm:almostsureCdim}
 There exists $t=t(p)$, such that
 \[\CC\dimh E(\omega) = t\,,\]
 almost surely, conditioned on $E(\omega)\neq\varnothing$.
\end{proposition}
\begin{proof}
	Recall that $E=E(\omega)=\Pi(T(\omega))$. Suppose that $\CC\dimh\Pi(T(\omega))<\alpha$. Then there exists a quasisymmetry $f\colon E\to Y$, where $Y$ is a metric space with $\dimh Y<\alpha$. Given $\iii\in\Lambda^*$, the restriction \[f|_{\Pi(T(\omega)\cap[\iii])}\colon\Pi(T(\omega)\cap[\iii])\to Y\] 
	is also a quasisymmetry (with the same control function). Since scaling preserves quasisymmetry as well, this implies that the event
	\[\CC\dimh\Pi(T(\omega))<\alpha\]
	is inherited, meaning that it holds for all finite trees and if it holds for a given tree $T$, it holds also for all the descendant sub-trees It follows from a standard zero-one law for inherited events (see \cite[Proposition 5.6]{LyonsPeres2016}) that
	\[\PP(\CC\dimh E<\alpha\,|\,E\neq\varnothing)\in\{0,1\}\,.\]
The proposition follows with
$t=\inf\{\alpha\,:\,\PP(\CC\dimh E<\alpha)>0\}$.	
\end{proof}

\begin{remark}
 \label{rem:embeddings}
 The proposition clearly holds for the global conformal dimension as well. (Recall the definition from the introduction.) Moreover, with the same technique, it is also easy to show that for a given (deterministic) $X$, it holds that $$\PP( X \text{ can be quasisymmetrically embedded to } E ) \in\{0,1\}$$ as well as $$\PP( E \text{ can be quasisymmetrically embedded to } X ) \in\{0,1\},$$ since the cases where $E=\varnothing$ are trivial.
\end{remark}

We will now start the proof of Theorem \ref{thm:notminimal} by constructing a (random) mapping $f\colon E\to\R^d$ that witnesses the estimate $\cdimh E<\dimh E$.
Consider a predetermined word of length $K\ge 1$ corresponding to one of the inner cubes, that is $\eta\in\Lambda^K$ such that $\eta_1\notin\Lambda_B$. Given $T(\omega)$, we define a new set $\widetilde{T}(\omega)$ by the following substitution rule:
\[
 \text{ If }\{ \iii|_{n-1} j \} \cap T_{n}(\omega)=\varnothing,\text{ for all $j\in\Lambda_B$, then substitute }i_{n}\to \eta i_{n}\,.
\]
Apply this for all $\iii\in T(\omega)$ and all $n\in\N$ and let $\widetilde{T}(\omega)$ be the resulting subset of $\Lambda^\N$. Let $\widetilde{\iii}$ denote the word obtained from $\iii$ after applying the aforementioned substitution rule for all $n$. In particular, the substitution is not applied iteratively, but it is applied to all $n$ at once: If we denote $\iii = i_1 i_2 i_3 \ldots \in T(\omega)$, then for any $n$, the single letter word $i_n$ becomes a word $\jjj_n\in\{i_n,\eta i_n\}$ in the substitution (depending on $T(\omega)$), and then $\widetilde\iii$ equals to $\jjj_1 \jjj_2 \jjj_3\ldots$. The definition of $\widetilde{\iii}$ clearly extends to finite words $\iii\in\Lambda^*$, as well (the substitution rule is applied for indices $1\le n\le|\iii|$ only). It is understood here that $\iii|_0$ is the empty word $\varnothing$ and, moreover, that $\widetilde\varnothing = \varnothing$. 

We define a map $f\colon E(\omega)\to F(\omega)$ by
\begin{equation}\label{eq:f}
f(\Pi(\iii))=\Pi( \widetilde{\iii} )\,,
\end{equation}
where $F(\omega) = f(E(\omega))$. Since a point $x\in E$ can have multiple representations, we should check that $f$ is well defined (whenever $E(\omega)\neq\varnothing$). So let $x\in E$ and $\iii,\jjj\in \Lambda^\N$ so that $\Pi(\iii) = \Pi(\jjj) = x$, but $\iii\neq\jjj$. We need to show that $\Pi(\widetilde\iii) = \Pi(\widetilde\jjj)$. Write $\kkk=|\iii\wedge\jjj|$, $\iii=\kkk\iii'$, and $\jjj=\kkk\jjj'$. Then $i_n,j_n\in\Lambda_B$ for all $n > |\kkk|+1$, and so either $\widetilde{\iii} = \widetilde{\kkk}\iii'$ and $\widetilde{\jjj} = \widetilde{\kkk}\jjj'$, or $\widetilde{\iii} = \widetilde{\kkk}\eta\iii'$ and $\widetilde{\jjj} = \widetilde{\kkk}\eta\jjj'$. In both cases, recalling \eqref{eq:Pi_decomposition},  it follows that $\Pi(\widetilde{\iii})=\Pi(\widetilde{\jjj})$.
This shows that $f(x)$ is indeed independent of the representation of $x$ as $\Pi(\iii)$. Note also that if $i_n \neq j_n$ for some $n$, then $\widetilde\iii \neq \widetilde\jjj$. This implies that $f$ is
one to one.

Note that due to the underlying random structure, $\widetilde{\iii\jjj}$ is not generally equal to $\tilde\iii\tilde\jjj$. To split the word $\widetilde{\iii\jjj}$ in a way that goes well with our substitution rule, we define
\[
 \prescript{}{\iii}{\widetilde{\jjj}} = \sigma^{|\widetilde{\iii}|}\widetilde{\iii\jjj}\,.
\]
This gives the identities $\widetilde{\iii\jjj} = \widetilde{\iii} \prescript{}{\iii}{\widetilde{\jjj}}$ and 
\begin{equation}
 \label{eq:tilde_splitting}
 \Pi(\widetilde{\iii\jjj})
 =
 \Pi(\widetilde\iii) + M^{-|\widetilde{\iii}|} \Pi( \prescript{}{\iii}{\widetilde{\jjj}} )
 =
 h_{ Q_{\widetilde\iii} } ( \Pi( \prescript{}{\iii}{\widetilde{\jjj}} ) ).
\end{equation}
For $\ppp \in T_n(\omega)$, write 
\begin{equation}\label{eq:f_l}
f_\ppp(\Pi(\iii)) = \Pi( \prescript{}{\ppp}{\widetilde{\iii}} )\,.
\end{equation} 
By \eqref{eq:tilde_splitting} we have
\begin{equation}
 \label{eq:f_commutator}
 f(x) = h_{Q_{\widetilde\ppp}} \circ f_\ppp \circ h^{-1}_{Q_\ppp} (x)
\end{equation}
for all $\ppp \in T_n(\omega)$ and $x \in \Pi([\ppp]) \cap E$. This allows us to separate the common prefix $\ppp$ from the equations, when considering $|f(\Pi(\ppp\iii))-f(\Pi(\ppp\jjj))|$, for example. Note that the mapping $f_\ppp$ is the random mapping that is obtained by replacing the tree $T(\omega)$ by the sub-tree rooted at $\ppp$ in the definition of $f$.

In order to prove Theorem \ref{thm:notminimal}, we need to prove the following lemmas:

\begin{lemma}\label{lem:dim_drop}
$\dim_H F<\dim_H E$ (almost surely, conditioned on non-extinction).
\end{lemma}

\begin{lemma}\label{lem:qs_eq}
The map $f\colon E\to F$ is a quasisymmetry (for all $\omega$).
\end{lemma}
Before going into the proofs, let us fix some notation for convenience: Let us replace the Euclidean distance by the maximum distance, which we keep denoting by $|x-y| =  \max_{i=1,...,d} \{|x_i-y_i| \}$. Since the maximum distance is bi-Lipschitz equivalent to the Euclidean distance, all properties related to dimensions and quasisymmetries are unaffected by this change. If $A$ and $B$ are variables, we denote $A\lesssim B$ if there is an absolute constant $C<\infty$ (depending only on $M$ and $d$) such that $ A\le C B$. If $A\lesssim B$ and $B\lesssim A$, then we write $A \approx B$.

\begin{proof}[Proof of Lemma \ref{lem:dim_drop}]
For $\iii\in\Lambda^*$, denote $E_\iii = \Pi([\iii] \cap T(\omega) )$. Then $E = \bigcup_{\iii\in\Lambda^n} E_\iii$ for any $n$. Moreover, $f(E_\iii)\subset Q_{\widetilde{\iii}}$ for all $\iii$, 
and thus, $\{Q_{\widetilde{\iii}}\}_{\iii\in T_n}$, $n\in\N$, constitute natural coverings for $F(\omega)$.

Fix $0<s<d$, and define random variables $Y^s_n(\omega) = \sum_{\iii\in T_n(\omega)} \diam\left(Q_{\widetilde{\iii}} \right)^s$. 
We will estimate $\EE\left( Y^s_n(\omega) \,|\,\mathcal{B}_n\right)$.
Fix $\iii\in\Lambda^n$ and condition on $\iii\in T_n$. Using the linearity of expectation,
\begin{align*}
&\EE\left(\sum_{\iii i\in T_{n+1}}\diam(Q_{\widetilde{\iii i}})^s \mid \iii\in T_n \right)\\
&=\diam(Q_{\widetilde{\iii}})^s\sum_{i=1}^{M^d}\PP(\iii i\in T_{n+1}\,|\,\iii\in T_n)\EE\left(\left(\frac{M^{-|\widetilde{\iii  i}|}}{M^{-|\widetilde{\iii}|}}\right)^s\,|\,\iii i\in T_{n+1}\right)\,.
\end{align*}
Let us denote the previous sum, $\sum_{i=1}^{M^d}\ldots$, by $S(\iii)$. Since $S(\iii)$ is invariant under the shift $\sigma (i_1\ldots i_k)=i_2\ldots i_k$ and $\PP(i\in T_1)=p$ for all $i\in\Lambda$, it follows that
\begin{align*}
S(\iii)&=\sum_{i=1}^{M^d}\PP(i\in T_1)\EE\left(M^{-s|\widetilde{i}|}\,|\,i\in T_1\right)\\
&=p\sum_{i\in\Lambda_B}\EE\left(M^{-s|\widetilde{i}|}\,|\,i\in T_1\right)+p\sum_{i\notin\Lambda_B}\EE\left(M^{-s|\widetilde{i}|}\,|\,i\in T_1\right)\,.
\end{align*}

For $i\in\Lambda_B$, we have $\EE\left(M^{-s|\widetilde{i}|}\,|\,i\in T_1\right)=M^{-s}$, whereas for $i\notin\Lambda_B$, 
we obtain (recall that $K=|\eta|$)
\begin{align*}
\EE\left(M^{-s|\widetilde{i}|}\,|\,i\in T_1\right)
&=
M^{-s}\PP(T_1\cap\Lambda_B\neq\varnothing)+M^{-s(K+1)}\PP(T_1\cap\Lambda_B=\varnothing)\\
&=
M^{-s}\left(1-(1-p)^{M^d-(M-2)^d}+M^{-sK}(1-p)^{M^d-(M-2)^d}\right)\\
&=
M^{-s}\left(1-(1-M^{-sK})(1-p)^{M^d-(M-2)^d}\right)\,.
\end{align*}

Let us denote
\[\kappa=\kappa(s,K)=1-\frac{(M-2)^d}{M^d}(1-M^{-sK})(1-p)^{M^d-(M-2)^d}\,,\]
and note that $\kappa<1$ if $s>0$. Based on the above computation, we may write
\begin{align}\label{eq:simf}
\EE\left(\sum_{\iii i\in T_{n+1}}\diam(Q_{\widetilde{\iii i}})^s\,|\,\iii\in T_n\right)=\diam(Q_{\widetilde{\iii}})^spM^{d-s}\kappa(s,K)\,.
\end{align}

Let $t$  be the unique solution to $pM^{d-t}\kappa(t,K)=1$ (note that this is well defined since $t\mapsto p M^{d-t}\kappa(t,K)$ is continuous and strictly decreasing, $p\kappa(d,K)<p<1$, and $p M^{-d}\kappa(0,K)=p M^{-d}>1$). Then
\eqref{eq:simf} yields
\begin{align*}
 \EE\left( Y^t_{n+1}(\omega) \mid T_n(\omega) \right)
 &=
 \sum_{\iii\in T_n(\omega)} \EE\left( \sum_{\iii i\in T_{n+1}(\omega)} \diam(Q_{\widetilde{\iii i}})^t \mid \iii\in T_n(\omega) \right) \\
 &=
 \sum_{\iii\in T_n(\omega)} \diam(Q_{\widetilde{\iii}})^t
  =
 Y^t_n(\omega)\,,
\end{align*}
and thus we have
$
 \EE\left( Y^t_{n+1} \,|\,\mathcal{B}_n\right)
  =
 Y^t_n
$.
By Doob's martingale convergence theorem \cite[Theorem 8.2]{Falconer1997}, 
almost surely $Y_n^t$ converges to a random variable $Y^t$, with $E(|Y^t|)<\infty$, and thus
\[
 \sup_n\sum_{\iii\in T_n}\diam(Q_{\widetilde{\iii}})^t
 =
 \sup_n Y_n^t(\omega)<\infty\,.
\]
and whence $\dim_H F(\omega)\le t$ almost surely. On the other hand, we know from \eqref{eq:dim_H_fp} that almost surely on non-extinction,
\[\dim_H E(\omega)=s\,,\]
where $s$ is the solution of $pM^{d-s}=1$. Using $\kappa(s,K)<1$, we have $t<s$ and whence the claim.
\end{proof}

Before proving Lemma \ref{lem:qs_eq}, we collect the main geometric properties of the map $f$ in the following proposition.

\begin{proposition}\label{prop:main_geom}
Let $x,y\in E$, with $x=\Pi(\iii)$, $y=\Pi(\jjj)$ and let 
$\kkk=\iii\wedge\jjj$. Denote $\iii=\kkk\iii'$ and $\jjj=\kkk\jjj'$.
 Then
\begin{equation}\label{eq:fxfy}
|f(x)-f(y)|\approx M^{-|\widetilde{\kkk}|}|\Pi(\iii')-\Pi(\jjj')|=M^{|\kkk|-|\widetilde{\kkk}|}|x-y|\,.
\end{equation}
\end{proposition}

\begin{proof}
The main geometric ingredient of the proof is the following observation: There is a constant $C>0$ such that the following holds for all $\ppp\in\Lambda^*$, $x\in Q_{\ppp}\cap E$ and $u\in\partial Q_{\ppp}$: Denote $u=\Pi(\ppp\mmm)$ and let $w=\Pi(\widetilde{\ppp}\mmm)$ (Note that $w$ equals $f(u)$, if $u\in E$, but if $u\notin E$, then $f(u)$ is not defined). Then 
\begin{equation}\label{eq:boundary_lip}
C^{-1} M^{|\ppp|-|\widetilde{\ppp}|}|x-u| \leq |f(x)-w|\leq C M^{|\ppp|-|\widetilde{\ppp}|}|x-u|\,.
\end{equation}

In order to prove \eqref{eq:boundary_lip}, write $x=\Pi(\ppp\ppp')$ and first consider $z=\Pi(\ppp')$ and $v=\Pi(\mmm)$. If $_\ppp\widetilde{\ppp'}=\ppp'$, we have $f_\ppp(z)-v=\Pi(\ppp')-v=z-v$. Otherwise, let $n>0$ be the smallest index such that $[\ppp\ppp'_{n-1}j]\cap T_{|\ppp|+n}(\omega)=\varnothing$ for all $j\in\Lambda_B$. Then we have the following estimates (see Figure \ref{fig:explaindisplayformulas}):
\begin{align}\label{kuva1}
z\in \overline{ Q_{\ppp'|_{n-1}}\setminus\bigcup_{j\in\Lambda_B}Q_{\ppp'|_{n-1}j}  }\,,
\end{align}
and
\begin{align}\label{kuva2}
f_\ppp(z)\in  Q_{\ppp'_{n-1}\eta}\subset \overline{ Q_{\ppp'|_{n-1}}\setminus\bigcup_{j\in\Lambda_B}Q_{\ppp'|_{n-1}j}  }\,,
\end{align}
and
\begin{align}\label{kuva3}
d\left(v,Q_{\ppp'|_{n-1}}\setminus\bigcup_{j\in\Lambda_B}Q_{\ppp'|_{n-1}j}\right)\ge M^{-n}\,.
\end{align}
giving the bounds $M^{-n} \le \min\{ |z-v|,|f_\ppp(z)-v| \}$ and $|f_\ppp(z)-z| \le M^{-n+1}$.
\begin{figure}
\centering
\begin{tikzpicture}[scale=1.2]
 \draw[thick] (-3,1.7) -- (-3,7.3);
 \coordinate (u) at (-3,3);
 \filldraw (u) circle (0.7pt);
 \node[right] at (u) {$v\in\partial [0,1]^d$};
\coordinate (r) at (2,2); 
\draw[pattern=north east lines,opacity=0.3] (r) rectangle ++ (5,5);
\node[right] (Ql) at (7,5) {$Q_{\ppp'|_{n-1}}$};
\fill[white] (r) ++ (1,1) rectangle ++ (3,3);
 \coordinate (x) at ($ (r) + (1.5,2.7) $); 
 \filldraw (x) circle (0.5pt);
 \node[right] at (x) {$z$};
 \draw[step=1,opacity=0.3] (r) ++ (1,1) grid ++ (3,3);
 \draw[|-|] (r) ++ (0,-0.2) -- node[below]{$M^{-n}$} ++ (1,0);
 \begin{scope}[scale=0.2, shift={(18,18)}]
 \draw[pattern=north east lines] (2,2) rectangle ++ (5,5);
 \draw[fill=white] (2,2) ++ (1,1) rectangle ++ (3,3);
 \draw[step=1] (2,2) ++ (1,1) grid ++ (3,3);
 \filldraw ($ (2,2) + (1.5,2.7) $) circle (0.5pt);
 \draw[<-] ($ (2,2) + (1.5,2.7) $) to [out=90,in=200] (7,10);
 \node[right] (fx) at (7,10) {$f_\ppp(z)$};
 \end{scope}
 \node[right] (Qleta) at (7,4) {$Q_{\ppp'|_{n-1}\eta}$};
 \draw[->] (7.1,4) to [out=120,in=20] (5.1,4.5);
\end{tikzpicture}
\caption{The formulas \eqref{kuva1}-\eqref{kuva3} are being illustrated here. Intuitively, $v$ is ``far'' from both $z$ and $f_\ppp(z)$ which in turn are ``near'' each other.}
\label{fig:explaindisplayformulas}
\end{figure}
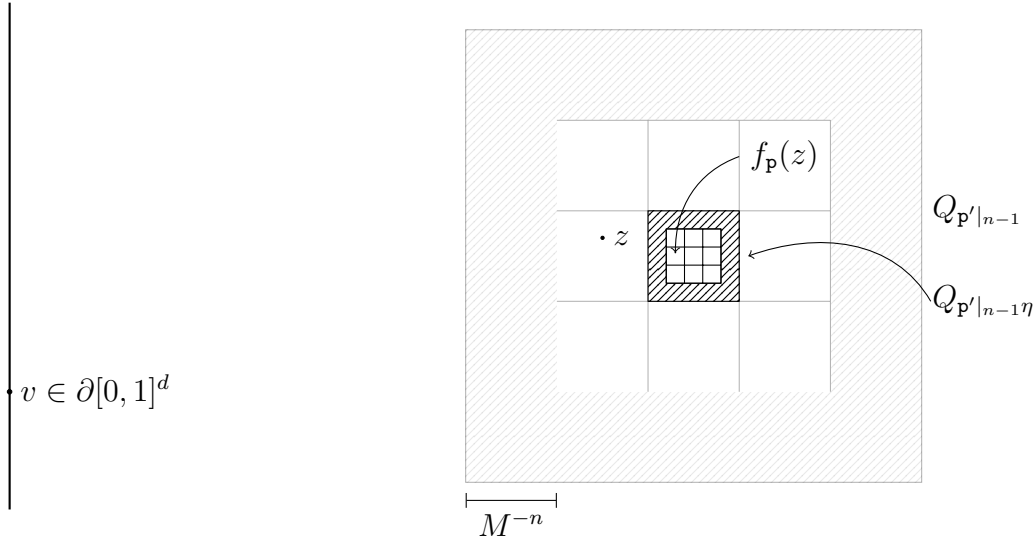
Thus we get
\[
 |z-v| \leq |f_\ppp(z) - z| + |f_\ppp(z)-v| \lesssim  |f_\ppp(z) - v| \leq |f_\ppp(z) - z| + | z -v | \lesssim | z - v |.
\]
In other words, $|f_\ppp(z)-v|\approx|z-v|$. Combining this with the identities
\begin{align*}
 |f(x)-w|=|\Pi(\widetilde{\ppp\ppp'})-\Pi(\widetilde\ppp\mmm)| &= M^{-|\widetilde\ppp|} |\Pi(_\ppp\widetilde\ppp')-\Pi(\mmm)| = M^{-|\widetilde\ppp|} |f_\ppp(z) - v|\,,\\
 |x-u| = |\Pi(\ppp\ppp')-\Pi(\ppp\mmm)| &= M^{-|\ppp|} |\Pi(\ppp')-\Pi(\mmm)| = M^{-|\ppp|} |z-v|\,,
\end{align*}
implies \eqref{eq:boundary_lip}.

Finally, let us derive the claim of the proposition from \eqref{eq:boundary_lip}. 
Let $u\in\partial Q_{\kkk i'_1}$ and $u'\in\partial Q_{\kkk j'_1}$ be the points along the line segment joining $x$ to $y$, and let $w$ and $w'$ be obtained from $u$ and $u'$ (respectively) as in \eqref{eq:boundary_lip}. 

Since either $\widetilde{\kkk i'_1}=\widetilde{\kkk}i'_1$ and $\widetilde{\kkk j'_1}=\widetilde{\kkk}j'_1$, or $\widetilde{\kkk i'_1}=\widetilde{\kkk}\eta i'_1$ and $\widetilde{\kkk j'_1}=\widetilde{\kkk}\eta j'_1$, it follows that either
$|w-w'|=M^{|\kkk|-|\widetilde{\kkk}|}|u-u'|$ or $|w-w'|=M^{|\kkk|-|\widetilde{\kkk}|-|\eta|}|u-u'|$. In any case,
\begin{equation}\label{eq:ww'}
|w-w'|\approx M^{|\kkk|-|\widetilde{\kkk}|}|u-u'|\,.
\end{equation}
Applying \eqref{eq:boundary_lip}, with $\ppp=\kkk i'_1,\kkk j'_1$, yields
\begin{align*}
|f(x)-f(y)|&\le |f(x)-w|+|w-w'|+|f(y)-w'|\\
&\approx M^{|\kkk i'_1|-|\widetilde{\kkk i'_1}|}|x-u|+M^{|\kkk|-|\widetilde{\kkk}|}|u-u'|+M^{|\kkk j'_1|-|\widetilde{\kkk j'_1}|}|u'-y|\\
&\approx M^{|\kkk|-|\widetilde{\kkk}|}|x-y|\,.
\end{align*}
To obtain the desired lower bound, we let $w\in\partial Q_{\widetilde{\kkk i'_1}}$, $w'\in\partial Q_{\widetilde{\kkk j'_1}}$ be the points along the line-segment joining $f(x)$ to $f(y)$ and let $u\in \partial Q_{\kkk i'_1}$, $u'\in\partial Q_{\kkk j'_1}$ be such that $w$ and $w'$ are obtained from $u$, $u'$ (respectively) as in \eqref{eq:boundary_lip}.  Then \eqref{eq:ww'} holds also for these $w,w',u,u'$.
Combining with \eqref{eq:boundary_lip}, we infer
\begin{align*}
|x-y|&\le |x-u|+|u-u'|+|y-u'|\\
&\approx M^{|\widetilde{\kkk}|-|\kkk|}|f(x)-w|+M^{|\widetilde{\kkk}|-|\kkk|}|w-w'|+M^{|\widetilde{\kkk}|-|\kkk|}|f(y)-w'|\\
&= M^{|\widetilde{\kkk}|-|\kkk|}|f(x)-f(y)|\,,
\end{align*}
as desired.
\end{proof}

\begin{proof}[Proof of Lemma \ref{lem:qs_eq}]
Let us show that the map $f$ defined by \eqref{eq:f} is a quasisymmetry $E(\omega)\to F(\omega)$. To that end, let $x,y,z\in E$, with $x=\Pi(\iii)$, $y=\Pi(\jjj)$, $z=\Pi(\ppp)$. Let $\kkk=\iii\wedge\jjj$.
We consider two different cases:

\emph{Case 1:} $\kkk \prec \ppp$. Let us denote $\ppp=\kkk\ppp'$ and $\iii=\kkk\iii'$. 
Let $\mmm=\iii'\wedge\ppp'$. It follows from \eqref{eq:fxfy} that
\begin{equation}\label{eq:xz}
|f(x)-f(z)|\approx M^{|\kkk\mmm|-|\widetilde{\kkk\mmm}|}|x-z|=M^{|\kkk|-|\widetilde{\kkk}|}M^{|\mmm|-|_\kkk\widetilde{\mmm}|}|x-z|\,.
\end{equation}
If $_\kkk\widetilde{\mmm}=\mmm$, then (using \eqref{eq:fxfy} for $x$ and $y$),
\[\frac{|f(x)-f(y)|}{|f(x)-f(z)|}\approx\frac{M^{|\kkk|-|\widetilde\kkk|}|x-y|}{M^{|\kkk|-|\widetilde\kkk|}|x-z|}\approx\frac{|x-y|}{|x-z|}\,.\]
Let us assume that $|_\kkk\widetilde{\mmm}|>|\mmm|$. In this case $|_\kkk\widetilde{\mmm}|=|\mmm|+N|\eta|$ for some $N\in\N$.  
Denoting $\mmm=m_1\ldots m_{|\mmm|}$, this means that $N$ of the symbols $m_i$, $i=1,\ldots,|\mmm|$ were replaced by $\eta m_i$. Thus, there is an index $0\le n\le |\mmm|-N$ such that 
\begin{equation}\label{eq:empty_bdr}
 [\kkk(\mmm|_{n})j]\cap T_{|\kkk|+n+1}(\omega)=\varnothing\text{ for all }j\in\Lambda_B\,.
\end{equation}
Since $x\in\Pi[\kkk\mmm]$ and $y\notin\Pi[\kkk(\mmm|_{n+1})]$, we obtain
\[
 |x-y|\ge d(x,\partial Q_{\kkk\mmm|_{n}})\ge M^{-|\kkk|-n-1}\ge M^{-\kkk-|\mmm|+N-1}\,.
\]
Putting this together with $|x-z|\le M^{-|\kkk|-|\mmm|}$ yields
\[M^{|_\kkk\widetilde{\mmm}|-|\mmm|}=M^{N|\eta|}\lesssim\left(\frac{|x-y|}{|x-z|}\right)^{|\eta|}\,,\]
so that recalling \eqref{eq:xz} and 
\[|f(x)-f(y)|\approx M^{|\kkk|-|\widetilde{\kkk}|}|x-y|\,,\]
we finally get,
\[
 \frac{|f(x)-f(y)|}{|f(x)-f(z)|}\approx M^{|_\kkk\widetilde{\mmm}|-|\mmm|}\frac{|x-y|}{|x-z|}\lesssim\left(\frac{|x-y|}{|x-z|}\right)^{|\eta|+1}\,.
\]
\emph{Case 2:} $\kkk\not\prec\ppp$. Denoting $\mmm=\kkk\wedge\ppp$ and $\kkk=\mmm\kkk'$, the estimate \eqref{eq:fxfy} gives
\[
 \frac{|f(x)-f(y)|}{|f(x)-f(z)|}
 \approx
 \frac{M^{|\kkk|-|\widetilde{\kkk}|}|x-y|}{M^{|\mmm|-|\widetilde{\mmm}|}|x-z|}
 =
 M^{|\kkk'|-|_\mmm\widetilde{\kkk'}|}\frac{|x-y|}{|x-z|}\le\frac{|x-y|}{|x-z|}\,.
\]
Thus, we have verified that $f$ is a quasisymmetry with control function $\eta(t)= C \max\{t,t^{|\eta|+1}\}$ with some constant $C<\infty$.
\end{proof}
\begin{figure}
\centering
\begin{tikzpicture}
\draw[thick] (0,-0.5) -- (0,6);
 \coordinate (y) at (0,3);
 \filldraw (y) circle (1pt);
 \node[left] at (y) {$y$};
 \coordinate (r) at (0,0); 
 \draw[pattern=north east lines,opacity=0.3] (r) rectangle ++ (5,5);
 \draw[fill=white] (r) ++ (1,1) rectangle ++ (3,3);
 \draw[step=1] (r) ++ (1,1) grid ++ (3,3);
 \coordinate (x) at ($ (r) + (1.5,2.7) $); 
 \filldraw (x) circle (1pt);
 \node[right] at (x) {$x$};
 \draw[|-|] (r) ++ (4,-0.2) -- node[below]{$M^{-|\kkk|-n-1}$} ++ (1,0);
 ;
 \node[] (Qi) at (1,5.5) {$Q_{\iii|_{|\kkk|+1}}$};
 \node[] (Qj) at (-1,5.5) {$Q_{\jjj|_{|\kkk|+1}}$};
 \node[above] (Qkm) at (4.5,4.9) {$Q_{\kkk(\mmm|_n)}$};
\end{tikzpicture}
\caption{In case (1) of the proof of Lemma \ref{lem:qs_eq}, $y$ can be in the boundary of $Q_{\iii|_{|\kkk|+1}}$ and thus also in $Q_{\kkk(\mmm|_n)}$, but due to 
	\eqref{eq:empty_bdr}, 
	it cannot fall into $Q_{\kkk(\mmm|_{n+1})}$.}
\label{fig:qs_proof}
\end{figure}
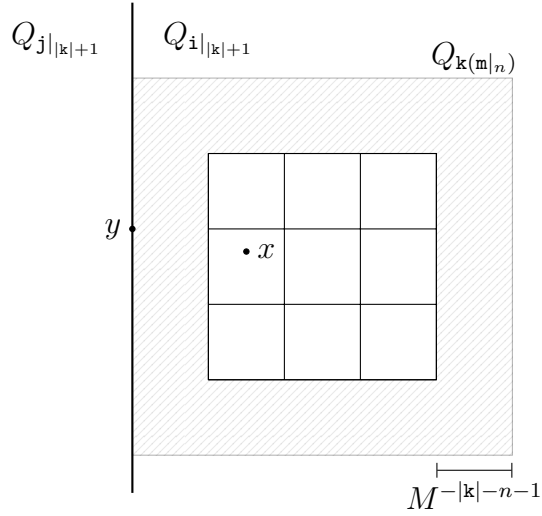

\section{Global conformal dimension}\label{sec:gcd}
In this Section,  we prove the global version of Theorem \ref{thm:notminimal}: 
\begin{theorem}\label{thm:glob}
Almost surely on non-extinction, the fractal percolation is not minimal for global conformal dimension.
\end{theorem}
We will show that the quasisymmetry $f\colon E(\omega)\to F(\omega)$, constructed in Section \ref{sec:cd}, may be extended to a global quasisymmetry $f\colon\R^d\to\R^d$. Clearly, it is enough to extend $f$ to a quasisymmetry on  $[0,1]^d$ such that $f$ is the identity on the boundary of $[0,1]^d$.

The definition of $f$ on $[0,1]^d\setminus E$ is slightly technical, but geometrically intuitive: We try to stretch the distances as little as possible on the complement of $E$. We will use an auxiliary map $g\colon[0,1]^d\to[0,1]^d$ with the following properties (For a cube $Q$, we denote by $\kappa Q$ the cube concentric with $Q$ and with side length $\kappa$ times the side length of $Q$):
\begin{enumerate}
\item $g$ is the identity map on $\partial[0,1]^d$.
\item $g$ maps $I = (1-\tfrac2M)[0,1]^d$  onto $(1-\tfrac2M)Q_\eta$ and $g|_I$ is a homothety (a scaling composed with a translation).
\item $g$ is bi-Lipschitz.
\end{enumerate} 
It is easy to construct such a map $g$, see Figure \ref{fig:extending_g}. For a cube $Q\in\mathcal{Q}_n$, let $g_Q=h_Q\circ g\circ (h_Q)^{-1}$. Note that each $g_Q$ is bi-Lipschitz with the same constants as $g$.
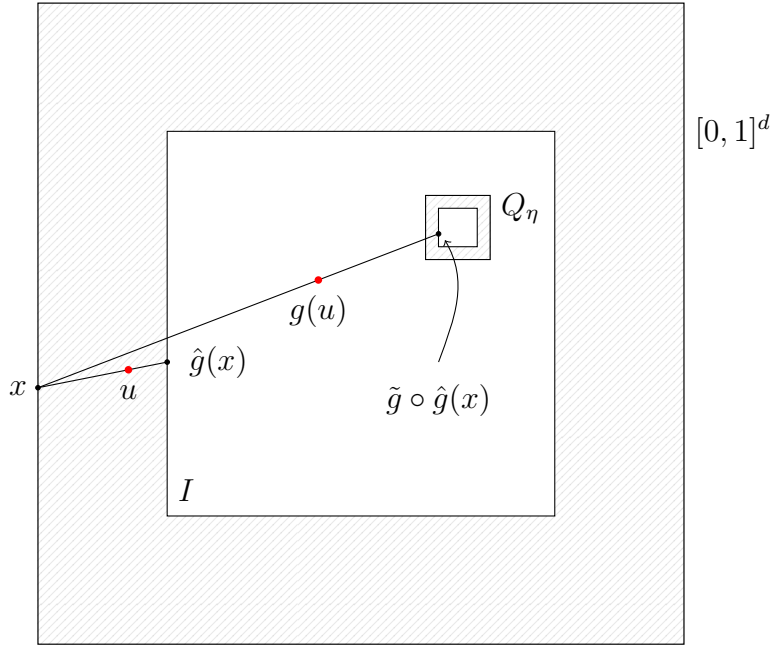
\begin{figure}
 \centering
 \begin{tikzpicture}[scale=1.7]
 \node[right] (Qi) at (5,4) {$[0,1]^d$};
 \draw[pattern=north east lines, opacity=0.3] (0,0) rectangle ++ (5,5);
 \draw (0,0) rectangle ++ (5,5);
 \draw[fill=white] (0,0) ++ (1,1) rectangle ++ (3,3);
 \node[right] (I) at (1,1.2) {$I$};
 \begin{scope}[shift={(3,3)}]
 \begin{scope}[scale=0.1]
 \node[right] (Qeta) at (5,4) {$Q_{\eta}$};
 \draw[pattern=north east lines, opacity=0.3] (0,0) rectangle ++ (5,5);
 \draw (0,0) rectangle ++ (5,5);
 \draw[fill=white] (0,0) ++ (1,1) rectangle ++ (3,3);
 \end{scope}
 \end{scope}
 \coordinate (x) at (0,2);
 \node[left] at (x) {$x$};
 \filldraw (x) circle (0.5pt);
 \coordinate (hatgx) at (1,2.2);
 \filldraw (hatgx) circle (0.5pt);
 \node[label=east:$\hat g(x)$] at (hatgx) {};
 \coordinate (ggx) at (3.1,3.2);
 \filldraw (ggx) circle (0.5pt);
 \draw[->] (ggx) ++ (0,-1) node[label=below:$\tilde g \circ\hat g(x)$]{} to [out=70,in=-60] ($ (ggx) + (0.05,-0.05) $);
 \draw[] (x) -- (hatgx);
 \draw[] (x) -- (ggx);
 \node [fill=red, circle,inner sep=1pt,label=below:$u$] (u) at ($ (x)!.7!(hatgx) $) {};
 \node [fill=red, circle,inner sep=1pt,label=below:$g(u)$] (gu) at ($ (x)!.7!(ggx) $) {};
\end{tikzpicture}
 \caption{
 One way to define $g$ is the following: Start with the homotheties $\hat g$ that maps $[0,1]^d$ to $I$, and $\tilde g$ that maps $[0,1]^d$ to $Q_\eta$. Let $g|_I = \tilde g|_I$. If $u\in [0,1]^d\setminus I$, then there are unique $x\in\partial[0,1]^d$ and $0\leq t < 1$ so that $u = (1-t)x + t \hat g(x)$. Set $g(u) = (1-t)x + t \tilde g \circ \hat g(x)$.
         }
 \label{fig:extending_g}
\end{figure}
We can now extend the map $f$ to $[0,1]^d\setminus E(\omega)$. Given $x=\Pi(\iii)\in [0,1]^d\setminus E(\omega)$, let $n\in\{0,1,2,\ldots\}$ be the largest index such that $\iii|_n\in T_n$. Write $\iii=\iii|_n\iii'=i_1\ldots i_n\iii'$.
%
If $T_{n+1}\cap\{\iii|_n j\,:\,j\in\Lambda_B\}\neq\varnothing$, we simply define 
 \[f(x)=\Pi(\widetilde{\iii_n}\iii')\]
%
Otherwise (if $T_{n+1}\cap\{\iii|_n j\,:\,j\in\Lambda_B\}=\varnothing$), we put
\[f(x)=g_{Q_{\widetilde{\iii|_n}}}(\Pi(\widetilde{\iii|_n}\iii'))\,.\]
Recall that if $n=0$, we interpret $\widetilde{\iii|_{n}}=\varnothing$, giving $f(x) = x$ in the first case and $f(x) = g(x)$ in the second case. On $E=E(\omega)$, we define $f$ as in the proof of Theorem \ref{thm:notminimal}. It is easy to check that $f$ is well defined, that is, independent of the representation $x=\Pi(\iii)$ (e.g. this follows from \eqref{eq:en_keksi_labelia} below). Moreover, it is an immediate consequence of the definition that $f(x)=x$, if $x\in\partial[0,1]^d$.

Given $\ppp\in T_n$, we also extend the ``blow-up'' maps $f_\ppp$ defined in \eqref{eq:f_l}. The maps $f_\ppp\colon[0,1]^d\to[0,1]^d$ are defined using the definition of $f$, but replacing $T(\omega)$ by the sub-tree rooted at $\ppp$. More precisely, we set
\[f_\ppp(x)=h^{-1}_{Q_{\widetilde{\ppp}}}\circ f\circ h_{Q_\ppp}(x)\,,\]
or in other words,
\begin{equation}
\label{eq:f_commutator_global}
 f(x) = h_{Q_{\widetilde{\ppp}}} \circ f_\ppp \circ h^{-1}_{Q_\ppp} (x)
\end{equation}
for all $x\in Q_\ppp$ and all $\ppp \in T_n$, $n\in\N$.

The following global version of Proposition \ref{prop:main_geom} describes the essential geometric properties of $f$.
\begin{proposition}\label{prop:glob}
Let $x = \Pi(\iii)$ and  $y=\Pi(\jjj)$ and let $\kkk=\iii\wedge\jjj$ and let $0\le n\le |\kkk|$ be the largest index such that $\kkk|_n\in T_n$. Let $\iii=\kkk|_n \iii'$, $\jjj=\kkk|_n \jjj'$ Then
\begin{equation}\label{eq:en_keksi_labelia}
|f(x)-f(y)|\approx M^{-|\widetilde{\kkk|_n}|}|\Pi(\iii')-\Pi(\jjj')|=M^{n-|\widetilde{\kkk|_n}|}|x-y|\,.
\end{equation}
\end{proposition}

\begin{proof}
The main ingredient of the proof is the following generalization of \eqref{eq:boundary_lip}: Given $x\in Q_{\ppp}$ and $u\in\partial Q_{\ppp}$, let $\ppp'$ be the longest beginning of $\ppp$ such that $\ppp'\in T_{|\ppp'|}$. Write
$u=\Pi(\ppp'\mmm)$ 
Then
\begin{equation}\label{eq:boundary_lip_glob}
|f(x)-f(u)|\approx M^{|\ppp'|-|\widetilde{\ppp'}|}|x-u|\,.
\end{equation}
Once \eqref{eq:boundary_lip_glob} is verified, \eqref{eq:en_keksi_labelia} is achieved following the proof of Proposition \ref{prop:main_geom}.

To verify \eqref{eq:boundary_lip_glob} we may assume that $x\notin E$, as otherwise the claim follows from \eqref{eq:boundary_lip}. Write $x=\Pi(\iii)=\Pi(\iii|_n\iii')$ where $n$ is the largest index such that $\iii|_n\in T_n$. We will consider two cases:

\emph{Case 1:} $|\ppp|\ge n$. Note that in this case $\ppp'=\iii|_n$. If $T_{n+1}\cap\{\iii|_n j\,:\,j\in\Lambda_B\}\neq\varnothing$, then
\[f(x)=\Pi(\widetilde{\iii|_n}\iii')\]
and
\[f(u)=\Pi(\widetilde{\iii|_n}\mmm)\,.\] 
Thus
\begin{align*}
|f(x)-f(u)|=M^{|\ppp'|-|\widetilde{\ppp'}|}|\Pi(\ppp\iii')-\Pi(\ppp\mmm)|=M^{|\ppp'|-|\widetilde{\ppp'}|}|x-u|\,.
\end{align*}

On the other hand, if $T_{n+1}\cap\{\iii|_n j\,:\,j\in\Lambda_B\}=\varnothing$, then
\[f(x)=g_{Q_{\widetilde{\iii|_{n}}}}(\Pi(\widetilde{\iii|_{n}}\iii'))\]
and
\[f(u)=g_{Q_{\widetilde{\iii|_{n}}}}(\Pi(\widetilde{\iii|_{n}}\mmm)\,.\] 
Since
\begin{align*}
|\Pi(\widetilde{\iii|_n}\iii')-\Pi(\widetilde{\iii|_n}\mmm)|= M^{|\ppp'|-|\widetilde{\ppp'}|}|x-u|\,,
\end{align*}
and $g_{Q_{\widetilde{\iii|_n}}}$ is Bi-Lipschitz (with constants that are independent of $\iii$ and $n$), we infer
\[
 |f(x)-f(u)|\approx M^{|\ppp'|-|\widetilde{\ppp'}|}|x-u|\,.
\]

\emph{Case 2:} $n>|\ppp|$. We argue as in the proof of Proposition \ref{prop:main_geom}. Note that in this case $\ppp'=\ppp$ and $f(u)=\Pi(\widetilde{\ppp}\mmm)$.

Let us first assume that $|\ppp|=0$, so that $f(u)=u=g(u)$. If $\widetilde{\iii|_n}=\iii|_n$, then 
$f(x)=x$ or $f(x)=g_{Q_{\iii|_n}}(x)$ depending on $T_{n+1}\cap\{\iii|_n j\,:\,j\in\Lambda_B\}$. In both cases, it is easy to see that 
\[|f(x)-f(u)|=|f(x)-u|\approx |x-u|\,.\]
If $\widetilde{\iii|_n}\neq\iii|_n$, there is $n'>0$ (with $n'\leq n$) such that 
\begin{align*}
x&\in \overline{ Q_{\iii|_{n'-1}}\setminus\bigcup_{j\in\Lambda_B}Q_{\iii'|_{n-1}j}  }\,,\\
f(x)&\in  Q_{\iii_{n'-1}\eta}\subset \overline{ Q_{\iii|_{n'-1}}\setminus\bigcup_{j\in\Lambda_B}Q_{\iii|_{n'-1}j}  }\,,
\end{align*}
and 
the estimate
\[
|f(x)-f(u)|=|f(x)-u|\approx |x-u|
\]
follows as in the proof of Proposition \ref{prop:main_geom} (recall \eqref{kuva1}--\eqref{kuva3}). 

If $|\ppp|>0$, the same argument applies for $f_\ppp$ yielding
\[|f_\ppp\circ h_{Q_\ppp}^{-1}(x)-f_\ppp\circ h_{Q_{\ppp}}^{-1}(u)|\approx|h^{-1}_{Q_\ppp}(x)-h^{-1}_{Q_\ppp}(u)|= M^{|\ppp|}|x-u|\,.\]
Thus, recalling \eqref{eq:f_commutator_global},
\begin{align*}
 |f(x)-f(u)|
 &=
 | h_{Q_{\widetilde{\ppp}}} \circ f_\ppp \circ h^{-1}_{Q_\ppp}(x) - h_{Q_{\widetilde{\ppp}}} \circ f_\ppp \circ h^{-1}_{Q_\ppp}(u) | \\
 &\approx
 M^{|\ppp|-|\widetilde\ppp|} | x-u |
\end{align*}
which proves \eqref{eq:boundary_lip_glob} (recall that $\ppp=\ppp'$).
\end{proof}
\begin{proof}[Proof of Theorem \ref{thm:glob}]
The claim is derived from Proposition \ref{prop:glob} similarly as Theorem \ref{thm:notminimal} is derived from Proposition \ref{prop:main_geom}.
\end{proof}

\begin{remarks}\label{rem:final_remarks}
\emph{a)} Observe that the upper bound that we obtain for the conformal and global conformal dimension of $E$ is quantitative in terms of $M$, $d$, and $p$. Indeed, since \eqref{eq:simf} is valid for any choice of $K\in\N$, letting $K\to\infty$ and setting
\[
 \kappa'=\lim_{K\to\infty}\kappa(s,K)=1-\frac{(M-2)^d}{M^d}(1-p)^{M^d-(M-2)^d}\,,
 \]
it follows that, almost surely on non-extinction,
\begin{align}\label{eq:gap}
\cdimh E\le d+\frac{\log p}{\log M}+\frac{\log\kappa'}{\log M}=\dimh(E)+\frac{\log \kappa'}{\log M}\,.
\end{align}
Most likely, the above estimate is very far from being optimal. It remains a challenging open problem to determine the exact value of $\cdimh E$ in terms of $p,M$ and $d$.

\emph{b)} A celebrated result (see \cite{Kovalev2006}) for conformal dimension is the fact that it does not take values in $(0,1)$. Using \eqref{eq:gap}, we can improve this for fractal percolation sets as follows: Given $M$ and $d$, there is a quantitative $\varepsilon=\varepsilon(M,d)>0$ such that, almost surely, fractal percolation sets with $\dimh E<1+\varepsilon$ satisfy $\cdimh E=0$. 
For instance, if $d=2$ and $M=3$, then solving 
\[\frac{\log p}{\log 3}+\frac{\log\kappa'}{\log 3}=-1\]
for $p$ implies that $\varepsilon(3,2)\approx 0.00389$. Some other values of $\varepsilon$ include, $\varepsilon(4,2)\approx0.00556$, $\varepsilon(5,2)\approx 0.00608$, $\varepsilon(3,3)\approx 0.00157$, $\varepsilon(4,3)\approx 0.00240$.

\emph{c)} Arguably, the most studied quantity for fractal percolation is the critical value $0<p_c(M,d)<1$ such that for $0<p<p_c$, the fractal percolation sets are almost surely totally disconnected, whereas for $p>p_c$, the set $E$ contains nontrivial connected components almost surely on non-extinction. See e.g. \cite{Don2015} and references therein. Note that our results hold for all $p\in(0,1)$. However, our method does not drop the dimension of the union of nontrivial connected components of $E$, but recall that this union has Hausdorff dimension $<\dimh E$, almost surely on non-extinction.

\emph{d)} Our main results can be generalized to many other random fractals constructed using the $M$-adic cubes. In particular, if the offspring distribution of $Q_\iii$, $\iii\in\Lambda^*$, is driven by a Galton-Watson process (see e.g. \cite{ChenOjalaRossiSuomala2017} for more details) and if for each $\iii\in\Lambda^n$, the probability for
\[T_{n+1}\cap \{\iii i\,:\,i\in\Lambda_B\}=\varnothing\text{ and }T_{n+1}\cap \{\iii i\,:\,i\in\Lambda\setminus\Lambda_B\}\neq\varnothing\,,\]
conditioned on $\iii\in T_{n}$, 
is $\ge c>0$, then the resulting random set $E$ satisfies $\CC\dimh E<\dimh E$, almost surely on $E\neq\varnothing$.
\end{remarks}

\end{document}